\documentclass{amsart}

\newtheorem{theorem}{Theorem}[section]

\newtheorem{example}{Example}[section]

\newtheorem{definition}{Definition}[section]

\numberwithin{equation}{section}

\newcommand{\abs}[1]{\left | #1\right |}

\begin{document}
\title{On Double Sequences }
\author{Richard F. Patterson* \& Huseyin Cakalli**\\*University of North Florida Jacksonville, Florida, USA\\**Maltepe University, Istanbul, Turkey}
\address{Department of Mathematics and Statistics,
University of North Florida, Jacksonville, Florida, 32224}
\email{rpatters@unf.edu}
\address{H\"Usey\.{I}n \c{C}akall\i\\
          Department of Mathematics, Maltepe University, Marmara E\u{g}\.{I}t\.{I}m K\"oy\"u, TR 34857, Maltepe, \.{I}stanbul-Turkey \; \; \; \; \; Phone:(+90216)6261050 ext:2248, fax:(+90216)6261113}
\email{huseyincakalli@maltepe.edu.tr ; hcakalli@gmail.com}

\date{\today}


\maketitle
\begin{abstract}
A double sequence $\{x_{k,l}\}$ is quasi-Cauchy if given an $\epsilon > 0$ there exists an $N \in {\bf N}$ such that
$$\max_{r,s= 1\mbox{ and/or } 0} \left \{|x_{k,l} - x_{k+r,l+s}|< \epsilon\right \} .$$ We study continuity type properties of factorable double functions defined on a double subset $A\times A$ of ${\bf R}^{2}$ into $\textbf{R}$, and obtain interesting results related to uniform continuity, sequential continuity, continuity, and a newly introduced type of continuity of factorable double functions defined on a double subset $A\times A$ of ${\bf R}^{2}$ into $\textbf{R}$.
\end{abstract}
\section{\bf Introduction}
In 1900, Pringsheim (\cite{Pring1}) introduced the concept of convergence of real double sequences. Four years later, Hardy (\cite{HardyOntheconvergenceofcertainmultipleseries}) introduced the notion of regular convergence for double sequences in the sense that double sequence has a limit in Pringsheim's sense and has one sided limits (see also \cite{robison, Ham1}).
%which guaranties limit of a double sequence in the first variable when the second is fixed, vice versa
A considerable number of papers which appeared in recent years study double sequences from various points of view (see \cite{AlotaibiandursaleenandAlghamdi,DuttaACharacterizationoftheClassofStatisticallyPre-CauchyDoubleSequencesofFuzzyNumbers,MursaleenandMohiuddineandSyed, PattersonandSavasAsymptoticEquivalenceofDoubleSequences, PattersonRH-regularTransformationsWhichSumsaGivenDoubleSequence, DjurcicandKocinacandZizovic, PattersonAtheoremonentirefourdimensionalsummabilitymethods, PattersonFourdimensionalmatrixcharacterizationP-convergencefieldsofsummabilitymethods}). Some results in the investigation are generalizations of known results concerning simple sequences to certain classes of double sequences, while other results reflect a specific nature of the Pringsheim convergence (e.g., the fact that a double sequence may converge without being bounded).

The aim of this paper is to introduce quasi-Cauchy double sequences, and investigate newly defined types of continuities for factorable double functions.

\section{\bf Priliminaries}
\begin{definition} (\cite{Pring1}) A double sequence ${\bf x}=\{x_{k,l}\}$ is Cauchy provided that, given an $\epsilon > 0$ there exists an $N \in {\bf N}$ such that $|x_{k,l} - x_{s,t}| < \epsilon$ whenever $k,l,s,t > N$.
\end {definition}

\begin{definition}
(\cite{Pring1}) A double sequence ${\bf x}=\{x_{k,l}\}$
has a {\bf Pringsheim limit} $L$ (denoted by P-$\lim x=L$) provided that, given an $\epsilon > 0$ there
exists an $N \in {\bf N}$ such that $\left|x_{k,l} - L\right| < \epsilon$
whenever $k,l > N$.  Such an $\bf x$ is described more briefly as ``$P$-convergent''.
\end{definition}
If $\lim |\textbf{x}| = \infty$ , (equivalently, for every $\varepsilon > 0$  there are $n_{1}, n_{2}\in{N}$ such that $| x_{m,n} |>M$  whenever $m>n_{1}$, $n>n_{2}$), then $\textbf{x}=\{x_{m,n}\}$ is said to be definitely divergent. A double sequence $\textbf{x}=\{x_{m,n}\}$ is bounded if there is an $M>0$  such that $|x_{m,n} |< M$ for all $m, n \in{\textbf{N}}$. Notice that a $P$-convergent double sequence need not be bounded.

\begin{definition}
(\cite{rfp}) A double sequence $\bf {y}$ is a {\bf double subsequence} of $\bf x$
provided that there exist increasing index sequences $\{n_{j}\}$
and $\{k_{j}\}$ such that, if $\{x_{j}\} = \{x_{n_{j},k_{j}}\}$, then $\bf y$ is
formed by
\[ \begin{array}{cccc}
  x_{1}&x_{2}& x_{5} &x_{10}\\
   x_{4}&x_{3}&x_{6}&-\\
 x_{9}&x_{8}&x_{7}&- \\
 - &-&-&-. \\
        \end{array}\]
\end{definition}

\section{{\bf Main Results}}

\begin{definition} \label{DefquasiCauchydoublesequence}
A double sequence $\textbf{x}=\{x_{k,l}\}$
is called quasi-Cauchy if given an $\epsilon > 0$ there
exists an $N \in {\bf N}$ such that $$\max_{r,s= 1\mbox{ and/or } 0} \left \{|x_{k,l} - x_{k+r,l+s}|< \epsilon\right \} $$
whenever $k,l > N$.
\end{definition}
Any $P$-convergent double sequence is quasi-Cauchy. Any Cauchy double sequence is quasi-Cauchy. Any subsequence of a $P$-convergent double sequence is $P$-convergent. Any subsequence of a Cauchy double sequence is Cauchy. But situation is different for quasi-Cauchy double sequences. There are subsequences of a quasi-Cauchy double sequence which are not quasi-Cauchy.

\begin{example} Double sequence defined by
let $s_{n}=\sum_{k=1}^{n}\frac{1}{k}$
\[ \begin{array}{ccccc}
s_{1}& s_{2}& s_{3}&s_{4}&\cdots\\
 s_{2}& s_{2}& s_{3}&s_{4}&\cdots\\
 s_{3}& s_{3}& s_{3}&s_{4}&\cdots\\
 s_{4}& s_{4}& s_{4}&s_{4}&\cdots\\
 \vdots& \vdots& \vdots&\vdots&\ddots\\
         \end{array}\]
is not $P$-convergent nor Cauchy, however it is a quasi-Cauchy double sequence In addition this double sequence has subsequences that are not quasi-Cauchy.
\end{example}

\begin{definition} \label{Defdoublesequentialcontinuity}
A factorable double function $f$ defined on a double subset $A\times A$ of ${\bf R}^{2}$ into $\textbf{R}$ is called double sequentially continuous at a point $L$ of  $A\times A$ if $f(\textbf{x})$ is $P$-convergent to $f(L)$ whenever $\textbf{x}=\{x_{k,l}\}$ is a $P$-convergent double sequence of points in $A\times A$ with P-limit $L$. If $f$ is double sequentially continuous at every point of  $A\times A$, we say $f$ is double sequentially continuous on $A\times A$.
\end{definition}

We note that any continuous function at a point $L$ of  $A\times A$ is also double sequentially continuous at $L$. The converse is also true:
\begin{theorem} \label{Theodoublesequentialcontinuityimpliescontinuity}
If a factorable double function $f$ defined on a double subset $A\times A$ of ${\bf R}^{2}$ is double sequentially continuous at $L$, then it is continuous.
\end{theorem}
\begin{proof} Suppose that $f$ is not continuous at $L$. Then there is an $\varepsilon_{0}>0$ such that for any $\delta > 0$ there exist an ${\bf x}_{\delta}$ so that
$|x_{1}(\delta)-L_{1}|<\delta$ and $|x_{2}(\delta)-L_{2}|<\delta$ but $|f(x_{1}(\delta), x_{2}(\delta))-f(L)|\geq \varepsilon_{0}$.
It is not difficult to construct a convergent double sequence with limit $L$ whose transformed sequence is not convergent to $f(L)$. Thus $f$ is not double sequentially continuous at $L$. This contradiction completes the proof of the theorem.
\end{proof}

Now we have obtained that a factorable double function $f$ defined on a double subset $A\times A$ of $R^{2}$ is double sequentially continuous at a point $L$ if and only if it is continuous. Using this equivalence we prove the following theorem.

\begin{theorem}
If a factorable double function $f$ defined on a double subset $A\times A$ of ${\bf R}^{2}$ preserves factorable double quasi-Cauchy sequences from $A\times A$, then it is continuous.
\end{theorem}

\begin{proof} Suppose that $f$ preserves factorable double quasi-Cauchy sequences from $A\times A$. Let $\boldsymbol{\alpha}=\{a_{i,j}\}$ be a double sequence defined by  \[ \begin{array}{ccccccc}
a_{1,1}& a_{1,2} & a_{1,3} &\cdots\\
      a_{2,1}& a_{2,2} & a_{2,3} &\cdots\\
     a_{3,1}& a_{3,2} &a_{3,3}&\cdots\\
    \vdots&\vdots& \vdots&\vdots& \vdots&\vdots&\ddots
         \end{array}\]
be any $P$-convergent factorable double sequence with P-limit $L$. Then the sequence \[ \begin{array}{ccccccc}
a_{1,1}\;L\;& a_{1,2}\;L\;& a_{1,3}\;L\; &...\\
L&L& L &L& L &L&...\\
     a_{2,1}\;L\;& a_{2,2}\;L\;& a_{2,3}\;L\; &\cdots\\
   L&L& L &L& L &L&...\\
   a_{3,1}\;L\;& a_{3,2}\;L\;& a_{3,3}\;L\;&\cdots\\
  L&L& L &L& L &L&...\\
    \vdots&\vdots& \vdots&\vdots& \vdots&\vdots&\ddots
         \end{array}\]
is also $P$-convergent with $P$-limit $L$. Since any convergent double sequence is quasi-Cauchy this sequence is quasi-Cauchy. So the transformed sequence  $f(\boldsymbol{\alpha})=\{f(a_{i,j})\}$ of the sequence $\boldsymbol{\alpha}$ is quasi-Cauchy. Thus it follows that

\[ \begin{array}{ccccccc}
f(a_{1,1})\;f(L)\;& f(a_{1,2})\;f(L)\;& f(a_{1,3})\;f(L)\; &...\\
f(L)&f(L)& f(L) &f(L)& f(L) &f(L)&...\\
     f(a_{2,1})\;f(L)\;& f(a_{2,2})\;f(L)\;& f(a_{2,3})\;f(L)\; &\cdots\\
   f(L)&f(L)& f(L) &f(L)& f(L) &f(L)&...\\
   f(a_{3,1})\;f(L)\;& f(a_{3,2})\;f(L)\;& f(a_{3,3})\;f(L)\;&\cdots\\
  f(L)&f(L)&f(L) &L& f(L) &f(L)&...\\
    \vdots&\vdots& \vdots&\vdots& \vdots&\vdots&\ddots
         \end{array}\]
is factorable quasi-Cauchy double sequence. Now it follows that $\{f(a_{i,j})\}$ is a $P$-convergent factorable double sequence with $P$-limit $f(L)$. By Theorem \ref{Theodoublesequentialcontinuityimpliescontinuity}, we get that the function $f$ is continuous. This completes the proof of the theorem.
\end{proof}

\begin{theorem} \label{Theodoubleseqofpairstendigtizeroimpliesdoubleforward}
Suppose that $I\times I$ is a two dimensional interval and
\[ \begin{array}{ccccccc}
a_{1,1}&b_{1,1}& a_{1,2} &b_{1,2}& a_{1,3} &b_{1,3}&\cdots\\
   d_{1,1}&c_{1,1}&d_{1,2}&c_{1,2}&d_{1,3}&c_{1,3}&\cdots\\
   a_{2,1}&b_{2,1}& a_{2,2} &b_{2,2}& a_{2,3} &b_{2,3}&\cdots\\
   d_{2,1}&c_{2,1}&d_{2,2}&c_{2,2}&d_{2,3}&c_{2,3}&\cdots\\
  a_{3,1}&b_{3,1}& a_{3,2} &b_{3,2}& a_{3,3} &b_{3,3}&\cdots\\
   d_{3,1}&c_{3,1}&d_{3,2}&c_{3,2}&d_{3,3}&c_{1,3}&\cdots\\
   \vdots&\vdots& \vdots&\vdots& \vdots&\vdots&\ddots
         \end{array}\]
 is a double sequence of ordered pairs in $I\times I$ with
$$\lim_{i}\abs{a_{i,i} -b_{i,i}}=\lim_{i}\abs{a_{i,i}-c_{i,i}}=\lim_{i}\abs{a_{i,i}-d_{i,i}}=0.$$  Then there exists a double quasi-Cauchy                                                                                      sequence $\{x_{i,j}\}$ with the property that for any ordered pair of integers $(i,j)$; $i,j>1$ there exists an ordered pair $(\bar{i},\bar{j})$;
$\bar{i},\bar{j}>1$ such that
$$(a_{i,j},b_{i,j})=(x_{\bar{i},\bar{j}},x_{\bar{i},\bar{j}+1})$$
$$(a_{i,j},c_{i,j})=(x_{\bar{i},\bar{j}},x_{\bar{i}+1,\bar{j}+1})$$
and
$$(a_{i,j},d_{i,j})=(x_{\bar{i},\bar{j}},x_{\bar{i}+1,\bar{j}}).$$
\end{theorem}
\begin{proof}
For every $(k,l)$; $k,l\geq 1$, fix
\[ \begin{array}{cccc}
  y^{k,l}_{0,0}&y^{k,l}_{0,1}& \cdots &y^{k,l}_{0,n_{l}}\\
  y^{k,l}_{1,0}&y^{k,l}_{1,1}& \cdots &y^{k,l}_{1,n_{l}}\\
\vdots&\vdots&\vdots&\vdots \\
  y^{k,l}_{m_{k},0}&y^{k,l}_{m_{k},1}& \cdots &y^{k,l}_{m_{k},n_{l}}\\
        \end{array}\]
in $I\times I$ with
$$y^{k,l}_{m_{k},0}=y^{k,l}_{0,n_{l}}=y^{k,l}_{m_{k},n_{l}}=a_{k+1,l+1},$$
$$y^{k,l}_{m_{k},2}=y^{k+1,l+1}_{0,0}=y^{k,l}_{m_{k},1}=b_{k+1,l+1},$$
$$y^{k+1,l+1}_{0,1}=y^{k+1,l+1}_{1,0}=y^{k+1,l+1}_{0,0}=c_{k+1,l+1},$$
and
$$y_{1,n_{l}}^{k,l}=y_{0,0}^{k+1,l+1}=y^{k,l}_{2,n_{l}}=d_{k+1,l+1}.$$
for $1\leq i\leq m_{k}$ and $1\leq j\leq n_{l}$ with
$$\abs{y^{k,l}_{i,j}-y^{k,l}_{i-1,j}}<\frac{1}{kl},$$
$$\abs{y^{k,l}_{i,j}-y^{k,l}_{i,j-1}}<\frac{1}{kl},$$
and
$$\abs{y^{k,l}_{i,j}-y^{k,l}_{i-1,j-1}}<\frac{1}{kl}.$$
Now the double sequence

\[ \begin{array}{cccccccccc}
a_{1,1}&b_{1,1}& y_{0,0}^{1,1} &\cdots& y_{0,n_{1}}^{1,1} &a_{1,1}&b_{1,1}&y^{1,1}_{0,0}&\cdots &y^{1,1}_{0,n_{1}}\\
d_{1,1}&c_{1,1}& y_{1,0}^{1,1} &\cdots& y_{1,n_{1}}^{1,1} &a_{1,1}&b_{1,1}&y^{1,1}_{1,0}&\cdots &y^{1,1}_{1,n_{1}}\\
y^{1,1}_{0,0}&y^{1,1}_{0,1}& y_{0,0}^{1,1} &\cdots& y_{0,n_{1}}^{1,1} &y^{1,1}_{0,0}&y^{1,1}_{0,1}&y^{1,1}_{0,0}&\cdots &y^{1,1}_{0,n_{1}}\\
\vdots&\vdots& \vdots &\vdots& \vdots &\vdots&\vdots&\vdots&\vdots &\vdots\\
y^{1,1}_{m_{1},0}&y^{1,1}_{m_{1},1}& y_{m_{1},0}^{1,1} &\cdots& y_{m_{1},n_{1}}^{1,1} &y^{1,1}_{m_{1},0}&y^{1,1}_{m_{1},1}&y^{1,1}_{m_{1},0}&\cdots &y^{1,1}_{m_{1},n_{1}}\\
a_{1,1}&b_{1,1}& y_{0,0}^{1,1} &\cdots& y_{0,n_{1}}^{1,1} &a_{2,2}&b_{2,2}&y^{2,2}_{0,0}&\cdots &y^{2}_{0,n_{1}}\\
d_{1,1}&c_{1,1}& y_{1,0}^{1,1} &\cdots& y_{1,n_{1}}^{1,1} &d_{2,2}&c_{2,2}&y^{2,2}_{1,0}&\cdots &y^{2,2}_{1,n_{2}}\\
y^{1,1}_{0,0}&y^{1,1}_{0,1}& y_{0,0}^{1,1} &\cdots& y_{0,n_{1}}^{1,1} &y^{2,2}_{0,0}&y^{2,2}_{0,1}&y^{2,2}_{0,0}&\cdots &y^{2,2}_{0,n_{2}}\\
\vdots&\vdots& \vdots &\vdots& \vdots &\vdots&\vdots&\vdots&\vdots &\vdots\\
y^{1,1}_{m_{1},0}& y^{1,1}_{m_{1},1}&y^{1,1}_{m_{1},0}&\cdots & y_{m_{1},1}&\vdots&\vdots&\vdots&\vdots &\vdots\\
\vdots&\vdots&\vdots&\vdots &\vdots&y^{2,2}_{m_{2},0}& y^{2,2}_{m_{2},1}&y^{2,2}_{m_{2},0}&\cdots& y^{2,2}_{m_{2},n_{2}}\\
         \end{array}.\]
is clearly a double sequence that has the desired property.
\end{proof}
\begin{theorem} \label{TheouniformcontinuitycoincideswithquasiCauchypreservence}
Suppose that $I\times I$ is any two dimensional interval. Then a two dimensional factorable real-valued function is uniformly continuous on $I\times I$ if and only if it is defined on $I\times I$ and preserves factorable double quasi-Cauchy sequences from $I\times I$.
\end{theorem}
\begin{proof}
It is clear that two dimensional uniformly continuous functions preserve double quasi-Cauchy sequence.

Conversely, suppose that $f$ defined on $I\times I$ in not uniformly continuous. Then there exists an $\epsilon >0$ such that for any $\delta>0$ there exist  $(a,b),(\bar{a},\bar{b})\in I\times I$ with $\sqrt{(a-\bar{a})^{2} + (b-\bar{b})^{2}}<\delta$
but  $\abs{f(a,b)-f(\bar{a},b)}\geq \epsilon$,  $\abs{f(a,b)-f(a,\bar{b})}\geq \epsilon$, and  $\abs{f(a,b)-f(\bar{a},\bar{b})}\geq \epsilon$, respectively.  Then by Theorem \ref{Theodoubleseqofpairstendigtizeroimpliesdoubleforward} there exists a double factorable quasi-Cauchy sequence $\bf {x}$$=\{x_{k}x_{l}\}$ such that for any ordered pair $(i,j)$ with $i\geq 1$ and $j\geq 1$, there exist ordered pairs integers $(\bar{i},\bar{j})$ with $a_{i,j}=x_{\bar{i},\bar{j}}$ and $b_{i,j}=x_{\bar{i}+1,\bar{j}+1}$. This implies that
$$\abs{f(x_{\bar{i}},x_{\bar{j}})-   f(x_{\bar{i}+1},x_{\bar{j}})   }\geq \epsilon ,$$
$$\abs{f(x_{\bar{i}},x_{\bar{j}})-   f(x_{\bar{i}},x_{\bar{j}+1})   }\geq \epsilon ,$$
and
$$\abs{f(x_{\bar{i}},x_{\bar{j}})-   f(x_{\bar{i}+1},x_{\bar{j}+1})   }\geq \epsilon.$$
Thus $\{f(x_{i},x_{j})\}$ is not quasi-Cauchy.  Thus $f$ does not preserve double quasi-Cauchy sequence.
\end{proof}

\begin{theorem}
Suppose that $f$ is a factorable double function defined on the bounded double interval $I\times I$. Then $f$ is uniformly continuous on $I\times I$ if and only if the image under $f$ of any Cauchy double sequence in $I\times I$ is  quasi-Cauchy.
\end{theorem}
\begin{proof}
By Theorem \ref{TheouniformcontinuitycoincideswithquasiCauchypreservence} if $f$ is a factorable uniformly continuous on $I\times I$ then the image of any quasi-Cauchy double sequence in $I\times I$ is quasi-Cauchy. Therefore the image of any double Cauchy under factorable function is quasi-Cauchy. Now let us establish the converse, to that end, suppose that the image of every Cauchy double sequence is quasi-Cauchy but the factorable to be uniformly continuous.
 Then there exists an $\epsilon >0$ such that for any $\delta>0$ there exist  $(x,y),(\bar{x},\bar{y})\in I\times I$ with $\sqrt{(x-\bar{x})^{2} + (y-\bar{y})^{2}}<\delta$
but  $\abs{f(x,y)-f(\bar{x},y)}\geq \epsilon$,  $\abs{f(x,y)-f(x,\bar{y})}\geq \epsilon$, and  $\abs{f(x,y)-f(\bar{x},\bar{y})}\geq \epsilon$, respectively.
For each $(m,n);$ $m,n\geq 1$, for fix double sequence $(x_{m},y_{n})$ and $(\bar{x}_{m},\bar{y}_{n})$ in $I\times I$
with $\sqrt{(x_{m}-\bar{x}_{m})^{2} + (y_{n}-\bar{y}_{n})^{2}}<\frac{1}{mn}$
but  $$\abs{f(x_{m},y_{n})-f(\bar{x}_{m},y_{n})}\geq \epsilon,$$  $$\abs{f(x_{m},y_{n})-f(x_{m},\bar{y}_{n})}\geq \epsilon,$$ and  $$\abs{f(x_{m},y_{n})-f(\bar{x}_{m},\bar{y}_{n})}\geq \epsilon,$$ respectively.
Since $I\times I$ is bounded there exists a $P$-convergent subsequence by a simple extension of Bolzano-Weierstrass theorem, say $\{x_{k,l}\}$.
The following double sequence
\[ \begin{array}{ccccccc}
x_{1,1}&y_{1,2}& x_{1,3} &y_{1,4}& x_{1,5} &y_{1,6}&\cdots\\
    y_{2,1}&x_{2,2}& y_{2,3} &x_{2,4}& y_{2,5} &x_{2,6}&\cdots\\
   x_{3,1}&y_{3,2}&x_{3,3}&y_{3,4}&x_{3,5}&y_{3,6}&\cdots\\
  y_{4,1}&x_{4,2}& y_{4,3} &x_{4,4}& y_{4,5} &x_{4,6}&\cdots\\
   x_{5,1}&y_{5,2}&x_{5,3}&y_{5,4}&x_{5,5}&x_{5,6}&\cdots\\
   \vdots&\vdots& \vdots&\vdots& \vdots&\vdots&\ddots
         \end{array}\]
         is $P$-convergent. Thus Cauchy, however the image
 \[ \begin{array}{ccccccc}
f(x_{1},x_{1})&f(y_{1},y_{2})&f( x_{1},x_{3}) &f(y_{1},y_{4})& f(x_{1},x_{5}) &f(y_{1},y_{6})&\cdots\\
    f(y_{2},y_{1})&f(x_{2},x_{2})& f(y_{2},y_{3}) &f(x_{2},f_{4})& f(y_{2},y_{5}) &f(x_{2},x_{6})&\cdots\\
   f(x_{3},x_{1})&f(y_{3},y_{2})&f(x_{3},x_{3})&f(y_{3},y_{4})&f(x_{3},x_{5})&f(y_{3},y_{6})&\cdots\\
  f(y_{4},y_{1})&f(x_{4},x_{2})&f( y_{4},y_{3}) &f(x_{4},x_{4})& f(y_{4},y_{5}) &f(x_{4},x_{6})&\cdots\\
   f(x_{5},x_{1})&f(y_{5},x_{2})&f(x_{5},x_{3})&f(y_{5},y_{4})&f(x_{5},x_{5})&f(x_{5},x_{6})&\cdots\\
   \vdots&\vdots& \vdots&\vdots& \vdots&\vdots&\ddots
         \end{array}\]
is no quasi-Cauchy. Thus we have a contradiction.
%$\abs{x_{m}-\bar{x}_{m}}<\frac{1}{m}$, $\abs{y_{n}-\bar{y}_{n}}<\frac{1}{n}$, and
\end{proof}

\section{{\bf General Metric Space}}
\begin{definition}
Suppose that $X\in {\bf R}^{2}$ is a set and $d: X\times X\rightarrow [0,\infty)$ is a function.
\begin{enumerate}
\item $d$ is called a pseudometric if it satisfies the following:
\begin{enumerate}
\item $d({\bf x},{\bf x})=0$ (i.e. $d((x_{1},y_{1}),(x_{1},y_{1}))=0$)
\item $d({\bf x},{\bf y})=d({\bf y},{\bf x})$
and
\item $d({\bf x},{\bf y})\leq d({\bf x},{\bf z})+d({\bf z},{\bf y})$ for all ${\bf x},{\bf y}$ and ${\bf z}\in X$

\end{enumerate}
\item $d$ is called a metric if $d$ also satisfies the following for all ${\bf x},{\bf y}\in X$
\begin{enumerate}
\item   $d({\bf x},{\bf y})<\infty$
\item $d({\bf x},{\bf y})=0$ implies that ${\bf x}={\bf y}$
\end{enumerate}
\end{enumerate}
\end{definition}

\begin{definition}
A metric space $(X,d)$ is called {\it non-incremental} if every quasi-Cauchy double sequence on $X$ is Cauchy
\end{definition}
\begin{definition}
An {\it ultrametric} space is a metric space $(X,d)$ which satisfies the following strengthening of the triangle inequality
$$d({\bf x},{\bf y})\leq \sup\left\{ d({\bf x},{\bf z}),d({\bf y},{\bf z})\right \}$$
for all ${\bf x},{\bf y},{\bf z}\in X$.
\end{definition}

\begin{theorem}
Ultarmetric spaces are non-incremental
\end{theorem}

\begin{proof}
Suppose that $(X,d)$ is an ultrametric space. Suppose that $\{x_{i,j}\}$ is a factorable quasi-Cauchy double sequence in $X$.
Suppose $\epsilon >0$ is given then for fix $K>0$ such that $m,n>K$  implies
$$d((x_{m},y_{n}), (x_{m+r},y_{n+s}))<\epsilon$$ where $r$ and $s$ are $1$ and/ or $0$.  Suppose that
$m,n<K$ with $m\leq n$. Then repeated application of the ultrametric inequality yields
\begin{eqnarray*}
 d({\bf x},{\bf y})&= &\sup\left \{\begin{array}{cccc}
  d((x_{m},y_{n}),(x_{m+1},y_{n}))& d((x_{m},y_{n}),(x_{m},y_{n+1}))&\cdots\\
  0& d((x_{m},y_{n}),(x_{m+1},y_{n+1}))&\cdots\\
d((x_{m+1},y_{n}),(x_{m+2},y_{n}))& d((x_{m+1},y_{n}),(x_{m+1},y_{n+1}))& \cdots\\
  0& d((x_{m+1},y_{n}),(x_{m+2},y_{n+1}))&\cdots\\
 d((x_{m+2},y_{n}),(x_{m+3},y_{n}))& d((x_{m+2},y_{n}),(x_{m+2},y_{n+1}))&\cdots\\
  0& d((x_{m+2},y_{n}),(x_{m+3},y_{n+1}))&\ddots\\
  \vdots&\vdots&\vdots \\

        \end{array}\right \}\\
        &=& \sup\left \{
   \begin{array}{ccccc}
 \epsilon&\epsilon&\epsilon&\epsilon& \cdots\\
0&\epsilon&0&\epsilon&\cdots\\
 \epsilon&\epsilon&\epsilon&\epsilon&\cdots \\
 0&\epsilon&0&\epsilon&\cdots\\
 \epsilon&\epsilon&\epsilon&\epsilon&\cdots \\
\vdots&\vdots&\vdots&\vdots&\ddots. \\
        \end{array}\right \} \leq \epsilon.
\end{eqnarray*}
Thus the factorable double sequence $\{x_{k,l}\}$ is also Cauchy.
\end{proof}\section{{\bf Conclusion}}
It is easy to see that double Cauchy sequences are double quasi-Cauchy. The converse is easily seen to be false as in the single dimensional case (\cite{david-Coleman}, \cite{CakalliForwardcontinuity}). One should also note that the single dimensional subsequences of an ordinary Cauchy sequence are also Cauchy sequence.  However the subsequence of quasi-Cauchy sequence is not quasi-Cauchy. But not just that, the subsequence of an ordinary Cauchy sequence is quasi-Cauchy. Now for double sequence the picture is very similar. Every subsequence of a double Cauchy sequence is also double quasi-Cauchy. The converse is also easily seen to be false.  Similar to ordinary sequences the subsequence of a double quasi-Cauchy sequence is arbitrary to say the least. That brings us to the starting point of this analysis. We illustrate this fact through Theorem 3.3.
One should also note that are nice connections between double quasi-Cauchy sequences and uniform continuity of two-dimensional real-valued functions. This is illustrated through the following theorem.
Suppose that $I\times I$ is any two dimensional interval. Then a two dimensional factorable real-valued functions is uniformly continuous on $I\times I$ if and only if it is defined on $I\times I$ and preserves factorable double quasi-Cauchy sequences from $I\times I$. Extensions and variations of the above theorems was also presented.

\bibliographystyle{amsplain}

\end{document}